\documentclass[10pt,twoside,12pt]{article}%
\usepackage{amssymb}
\usepackage{amsfonts}
\usepackage{amsmath}
\usepackage{graphicx}
\usepackage{latexsym}
\usepackage[a4paper,hmargin=2.5cm,vmargin=2.5cm]{geometry}%
\newtheorem{theorem}{Theorem}[section]

\newtheorem{lemma}{Lemma}[section]
\newtheorem{proposition}{Proposition}[section]
\newtheorem{remark}{Remark}[section]
\newtheorem{example}{Example}[section]
\newtheorem{corollary}{Corollary}[section]
\newenvironment{proof}[1][Proof]{\noindent\textbf{#1.} }{\ \rule{0.5em}{0.5em}}
\begin{document}

\title{Sensitivity of a nonlinear ordinary BVP\\with fractional Dirichlet-Laplace operator}
\author{\textbf{{\large Dariusz Idczak}}\vspace{1mm}\\\textit{{\small Faculty of Mathematics and Computer Science}}\\\textit{{\small University of Lodz}}\\\textit{{\small 90-238 Lodz, Banacha 22}}\\\textit{{\small Poland}}\\{\small e-mail: idczak@math.uni.lodz.pl}}
\date{{\small \ }}
\maketitle

\begin{abstract}
In the paper, we derive a sensitivity result for a nonlinear fractional
ordinary elliptic system on a bounded interval with Dirichlet boundary
conditions. More precisely, using a global implicit function theorem, we show
that, for any functional parameter, there exists a unique solution to such a
problem and dependence of solutions on functional parameters is continuously differentiable.

\end{abstract}

\textit{{\footnotesize Key
words}}. {\scriptsize Fractional Dirichlet-Laplace operator, Stone-von Neumann
operator calculus, global implicit function theorem, Palais-Smale condition}

\section{\textbf{Introduction}}

\setcounter{equation}{0}In our paper, we study a nonlinear ordinary boundary
value problem on the interval $(0,\pi)$, involving a Dirichlet-Laplace
operator $(-\Delta)^{\beta}$ of order $\beta>\frac{1}{2}$,%
\begin{equation}
(-\Delta)^{\beta}x(t)=f(t,x(t),u(t)),\ t\in(0,\pi)\text{ a.e.,} \label{W}%
\end{equation}
where $(-\Delta):H_{0}^{1}\cap H^{2}\rightarrow L^{2}$ is the
Dirichlet-Laplace operator, $f:(0,\pi)\times\mathbb{R}^{m}\times\mathbb{R}%
^{r}\rightarrow\mathbb{R}^{m}$, $x$ is an unknown function and $u$ - a
functional parameter.

Problems involving fractional Laplacians are extensively investigated in
resent years due to their numerous applications, among others in probability,
fluid mechanics, hydrodynamics (see, for example, \cite{Barrios},
\cite{BogBycz2}, \cite{CabreTan}, \cite{Caff}, \cite{Vaz} and references therein).

Definition of the fractional Laplacian adopted in our paper comes from the
Stone-von Neumann operator calculus and is based on the spectral integral
representation theorem for a self-adjoint operator in Hilbert space. It
reduces to a \textit{series}\ form which is taken by other authors as a
definition (\cite{Barrios}, \cite{Bors2}, \cite{CabreTan}). Our more general
approach allows us to obtain useful properties of this fractional operator in
a smart way. This approach has also been used in \cite{Id3}.

In the first part of the paper, we recall some facts from the theory of
spectral integral and Stone-von Neumann operator calculus. Next, we derive
some properties of positive powers of the ordinary Dirichlet-Laplace operator
and their domains (among other, some embedding theorems). In the second part,
we use a global implicit function theorem (\cite{Id}, \cite{Id2}) to prove
existence and uniqueness of a solution to problem (\ref{W}) as well as its
sensitivity. By sensitivity we mean continuous differentiability of the
mapping%
\[
u\longmapsto x_{u}%
\]
where $x_{u}$ is a unique solution to the problem, corresponding to a
parameter $u$. This property can be used in optimal control for system
(\ref{W}).

Similar method but based on a global diffeomorphism theorem (\cite{ISW}) and
applied to a nonlinear integral Hammerstein equation is presented in
\cite{Bors1} with an application to the problem%
\[
\lambda(-\Delta)^{\frac{\sigma}{2}}x(t)+h(t,x(t))=(-\Delta)^{\frac{\sigma}{2}%
}z(t),\ t\in(0,1),
\]
with the exterior Dirichlet boundary condition%
\[
x(t)=0,\ t\in(-\infty,-1]\cup\lbrack1,\infty).
\]
In \cite{Bors2}, the problem of type (\ref{W}) on a bounded Lipschitzian
domain $\Omega\subset\mathbb{R}^{n}$ ($n\geq2$) and with an exterior Dirichlet
boundary condition, is studied. Continuous dependence of solutions on
parameters (stability) is investigated therein.

In \cite{Id3}, using a variational method, we derive an existence result for
the so-called bipolynomial fractional Dirichlet-Laplace problem%
\[%
{\displaystyle\sum\nolimits_{i,j=0}^{k}}
\alpha_{i}\alpha_{j}[(-\Delta)_{\omega}]^{\beta_{i}+\beta_{j}}u(x)=D_{u}%
F(x,u(x)),\ x\in\Omega\text{ a.e.},
\]
where $\alpha_{i}>0$ for $i=0,...,k$ ($k\in\mathbb{N}\cup\{0\}$) and
$0\leq\beta_{0}<\beta_{1}<...<\beta_{k}$, $(-\Delta)_{\omega}:D((-\Delta
)_{\omega})\subset L^{2}\rightarrow L^{2}$ is a weak Dirichlet-Laplace
operator, $\Omega\subset\mathbb{R}^{N}$ is a bounded open set, $F:\Omega
\times\mathbb{R}\rightarrow\mathbb{R}$, $D_{u}F$ is the partial derivative of
$F$ with respect to $u$.

\section{Integral representation of a self-adjoint operator}

Results presented in this section can be found, in the case of complex Hilbert
space, for example in \cite{Alex}, \cite{Mlak}. Their proofs can be moved
without any or with small changes to the case of real Hilbert space. In the
next, we shall deal only with real Hilbert spaces. Such a preliminary section
has also been included in the paper \cite{Id3}.

Let $H$ be a real Hilbert space with a scalar product $\left\langle
\cdot,\cdot\right\rangle :H\times H\rightarrow\mathbb{R}$. Let us denote by
$\Pi(H)$ the set of all projections of $H$ on closed linear subspaces and by
$\mathcal{B}$ - the $\sigma$-algebra of Borel subsets of $\mathbb{R}$. By the
spectral measure in $\mathbb{R}$ we mean a set function $E:\mathcal{B}%
\rightarrow\Pi(H)$ that satisfies the following conditions:

\begin{itemize}
\item[$\cdot$] for any $x\in H$, the function%
\begin{equation}
\mathcal{B}\ni P\longmapsto E(P)x\in H \label{vm}%
\end{equation}
is a vector measure

\item[$\cdot$] $E(\mathbb{R})=I$

\item[$\cdot$] $E(P\cap Q)=E(P)\circ E(Q)$ for $P,Q\in\mathcal{B}$.
\end{itemize}

By a support of a spectral measure $E$ we mean the complement of the sum of
all open subsets of $\mathbb{R}$ with zero spectral measure.

If $b:\mathbb{R}\rightarrow\mathbb{R}$ is a bounded Borel measurable function,
defined $E$ - a.e., then the integral $%
{\displaystyle\int\nolimits_{-\infty}^{\infty}}
b(\lambda)E(d\lambda)$ is defined by%
\[
(%
{\displaystyle\int\nolimits_{-\infty}^{\infty}}
b(\lambda)E(d\lambda))x=%
{\displaystyle\int\nolimits_{-\infty}^{\infty}}
b(\lambda)E(d\lambda)x
\]
for any $x\in H$ where the integral $%
{\displaystyle\int\nolimits_{-\infty}^{\infty}}
b(\lambda)E(d\lambda)x$ (with respect to the vector measure) is defined in a
standard way, with the aid of the sequence of simple functions converging
$E(d\lambda)x$ - a.e. to $b$ (cf. \cite{Alex}).

If $b:\mathbb{R}\rightarrow\mathbb{R}$ is an unbounded Borel measurable
function, defined $E$ - a.e., then, for any $x\in H$ such that
\begin{equation}%
{\displaystyle\int\nolimits_{-\infty}^{\infty}}
\left\vert b(\lambda)\right\vert ^{2}\left\Vert E(d\lambda)x\right\Vert
^{2}<\infty\label{fm}%
\end{equation}
(the above integral is taken with respect to the nonnegative measure
$\mathcal{B\ni}P\longmapsto\left\Vert E(P)x\right\Vert ^{2}\in\mathbb{R}%
_{0}^{+}$), there exists the limit%
\[
\lim%
{\displaystyle\int\nolimits_{-\infty}^{\infty}}
b_{n}(\lambda)E(d\lambda)x
\]
of integrals (with respect to the vector measure (\ref{vm})) where
\[
b_{n}:\mathbb{R\ni\lambda\longmapsto}\left\{
\begin{array}
[c]{ccc}%
b(\lambda) & \text{as} & \left\vert b(\lambda)\right\vert \leq n\\
0 & \text{as} & \left\vert b(\lambda)\right\vert >n
\end{array}
\right.  .
\]
Let us denote the set of all points $x$ with property (\ref{fm}) by $D$. One
proves that $D$ is dense linear subspace of $H$ and by $%
{\displaystyle\int\nolimits_{-\infty}^{\infty}}
b(\lambda)E(d\lambda)$ one denotes the operator%
\[%
{\displaystyle\int\nolimits_{-\infty}^{\infty}}
b(\lambda)E(d\lambda):D\subset H\rightarrow H
\]
given by%
\[
(%
{\displaystyle\int\nolimits_{-\infty}^{\infty}}
b(\lambda)E(d\lambda))x=\lim%
{\displaystyle\int\nolimits_{-\infty}^{\infty}}
b_{n}(\lambda)E(d\lambda)x.
\]
Of course, $D=H$ and
\[
\lim%
{\displaystyle\int\nolimits_{-\infty}^{\infty}}
b_{n}(\lambda)E(d\lambda)x=%
{\displaystyle\int\nolimits_{-\infty}^{\infty}}
b(\lambda)E(d\lambda)x
\]
when $b:\mathbb{R}\rightarrow\mathbb{R}$ is a bounded Borel measurable
function, defined $E$ - a.e.

For $x\in D$, we have%
\[
\left\Vert (%
{\displaystyle\int\nolimits_{-\infty}^{\infty}}
b(\lambda)E(d\lambda))x\right\Vert ^{2}=%
{\displaystyle\int\nolimits_{-\infty}^{\infty}}
\left\vert b(\lambda)\right\vert ^{2}\left\Vert E(d\lambda)x\right\Vert ^{2}.
\]
Moreover,%
\begin{equation}
(%
{\displaystyle\int\nolimits_{-\infty}^{\infty}}
b(\lambda)E(d\lambda))^{\ast}=%
{\displaystyle\int\nolimits_{-\infty}^{\infty}}
b(\lambda)E(d\lambda), \label{samosp}%
\end{equation}
i.e. the operator $%
{\displaystyle\int\nolimits_{-\infty}^{\infty}}
b(\lambda)E(d\lambda)$ is self-adjoint.

\begin{remark}
\label{sense}To integrate a Borel measurable function $b:B\rightarrow
\mathbb{R}$ where $B$ is a Borel set containing the support of the measure
$E$, it is sufficient to extend $b$ on $\mathbb{R}$ to a whichever Borel
measurable function (putting, for example, $b(\lambda)=0$ for $\lambda\notin
B$).
\end{remark}

If $b:\mathbb{R}\rightarrow\mathbb{R}$ is Borel measurable and $\sigma
\in\mathcal{B}$, then by the integral
\[%
{\displaystyle\int\nolimits_{\sigma}}
b(\lambda)E(d\lambda)
\]
we mean the integral
\[%
{\displaystyle\int\nolimits_{-\infty}^{\infty}}
\chi_{\sigma}(\lambda)b(\lambda)E(d\lambda)
\]
where $\chi_{\sigma}$ is the characteristic function of the set $\sigma$
(\footnote{Integral $%
{\displaystyle\int\nolimits_{\sigma}}
b(\lambda)E(d\lambda)$ can be also defined with the aid of the restriction of
$E$ to the set $\sigma$}).

Next theorem plays the fundamental role in the spectral theory of self-adjoint operators.

\begin{theorem}
\label{main}If $A:D(A)\subset H\rightarrow H$ is self-adjoint and the
resolvent set $\rho(A)$ is non-empty, then there exists a unique spectral
measure $E$ with the closed support $\Lambda=\sigma(A)$, such that%
\[
A=%
{\displaystyle\int\nolimits_{-\infty}^{\infty}}
\lambda E(d\lambda)=%
{\displaystyle\int\nolimits_{\sigma(A)}}
\lambda E(d\lambda).
\]

\end{theorem}

The basic notion in the Stone-von Neumann operator calculus is a function of a
self-adjoint operator. Namely, if $A:D(A)\subset H\rightarrow H$ is
self-adjoint and $E$ is the spectral measure determined according to the above
theorem, then, for any Borel measurable function $b:\mathbb{R}\rightarrow
\mathbb{R}$, one defines the operator $b(A)$ by%
\[
b(A)=%
{\displaystyle\int\nolimits_{-\infty}^{\infty}}
b(\lambda)E(d\lambda)=%
{\displaystyle\int\nolimits_{\sigma(A)}}
b(\lambda)E(d\lambda).
\]
It is known that the spectrum $\sigma(b(A))$ of $b(A)$ is given by%
\begin{equation}
\sigma(b(A))=\overline{b(\sigma(A))} \label{ow}%
\end{equation}
provided that $b$ is continuous (it is sufficient to assume that $b$ is
continuous on $\sigma(A)$).

We have the following general result.

\begin{proposition}
\label{zloz}If $b,d:\mathbb{R}\rightarrow\mathbb{R}$ are Borel measurable
functions and $E$ is the spectral measure for a self-adjoint operator
$A:D(A)\subset H\rightarrow H$ with non-empty resolvent set, then
\[
(b\cdot d)(A)\supset b(A)\circ d(A)
\]
and%
\begin{equation}
(b\cdot d)(A)=b(A)\circ d(A) \label{rowbd}%
\end{equation}
if and only if%
\[
D((b\cdot d)(A))\subset D(d(A)).
\]

\end{proposition}

Using the above proposition one can deduce that, for any $n\in\mathbb{N}$,
$n\geq2$, and a Borel measurable function $b:\mathbb{R}\rightarrow\mathbb{R}$,%
\begin{equation}
(b(A))^{n}=b^{n}(A). \label{formulan}%
\end{equation}
When $b(\lambda)=\lambda$, (\ref{formulan}) gives%
\begin{equation}
A^{n}=%
{\displaystyle\int\nolimits_{-\infty}^{\infty}}
\lambda^{n}E(d\lambda). \label{formulann}%
\end{equation}
If $n=1$, then (\ref{formulann}) follows from Theorem \ref{main}. Since
$E(\mathbb{R})=I$, therefore the identity operator $I$ can be written as%
\[
I=%
{\displaystyle\int\nolimits_{-\infty}^{\infty}}
1E(d\lambda).
\]
If $\beta>0$, then formula (\ref{formulan}) with
\[
b:\mathbb{R}\ni\lambda\rightarrow\left\{
\begin{array}
[c]{ccc}%
0 & ; & \lambda<0\\
\lambda^{\frac{\beta}{2}} & ; & \lambda\geq0
\end{array}
\right.
\]
and $n=2$ implies the following proposition (cf. Remark \ref{sense}).

\begin{proposition}
\label{zlozpot}If $\sigma(A)\subset\lbrack0,\infty)$, then%
\begin{equation}
A^{\frac{\beta}{2}}\circ A^{\frac{\beta}{2}}=A^{\beta}. \label{beta1beta2}%
\end{equation}

\end{proposition}

\section{Fractional Dirichlet-Laplace operator}

Consider the one-dimensional Dirichlet-Laplace operator on the interval
$(0,\pi)$
\[
-\Delta:H_{0}^{1}\cap H^{2}\subset L^{2}\rightarrow L^{2}%
\]
given by
\[
-\Delta x(t)=-x^{^{\prime\prime}}(t)
\]
where $H_{0}^{1}=H_{0}^{1}((0,\pi),\mathbb{R}^{m})$, $H^{2}=H^{2}%
((0,\pi),\mathbb{R}^{m})$ are classical Sobolev spaces and $L^{2}=L^{2}%
((0,\pi),\mathbb{R}^{m})$. In an elementary way, one can check that this
operator is self-adjoint,
\[
\sigma(-\Delta)=\sigma_{p}(-\Delta)=\{j^{2};\ j\in\mathbb{N}\}
\]
($\sigma_{p}(-\Delta)$ is the pointwise spectrum of $(-\Delta)$) and the
eigenspace $N(j^{2})$ corresponding to the eigenvalue $\lambda_{j}=j^{2}$ is
the set $\{c\sin jt;\ c\in\mathbb{R}^{m}\}$. The system of functions%
\[
e_{j,i}=(0,...,0,\underset{i-th}{\underbrace{\sqrt{\frac{2}{\pi}}\sin jt}%
},0,...,0),\ j=1,2,...,\ i=1,...,m,
\]
is the hilbertian basis (complete ortonormal system) in $L^{2}$.

Now, let us fix any $\beta>0$ and consider the operator%
\[
(-\Delta)^{\beta}:D((-\Delta)^{\beta})\subset L^{2}\rightarrow L^{2}%
\]
where%
\begin{align}
D((-\Delta)^{\beta})  &  =\{x(t)\in L^{2};\ \nonumber\\%
{\displaystyle\int\nolimits_{\sigma(-\Delta)}}
\left\vert \lambda^{\beta}\right\vert ^{2}\left\Vert E(d\lambda)x\right\Vert
^{2}  &  =%
{\displaystyle\sum\nolimits_{j=1}^{\infty}}
((j^{2})^{\beta})^{2}\left\vert a_{j}\right\vert ^{2}<\infty\\
\text{where }x(t)  &  =(%
{\displaystyle\int\nolimits_{\sigma(-\Delta)}}
1E(d\lambda)x)(t)=%
{\displaystyle\sum\nolimits_{j=1}^{\infty}}
a_{j}\sqrt{\frac{2}{\pi}}\sin jt\}\nonumber
\end{align}
(here $E$ is the spectral measure given by Theorem \ref{main} for the operator
$(-\Delta)$, $a_{j}\sqrt{\frac{2}{\pi}}\sin jt$ is the projection of $x$ on
the $m$-dimensional eigenspace $N(j^{2})$ of the operator $(-\Delta)$ and%
\begin{multline*}
(-\Delta)^{\beta}x(t)=((%
{\displaystyle\int\nolimits_{\sigma(-\Delta)}}
\lambda^{\beta}E(d\lambda))x)(t)=(\lim%
{\displaystyle\int\nolimits_{\sigma(-\Delta)}}
(\lambda^{\beta})_{n}E(d\lambda)x)(t)\\
=%
{\displaystyle\sum\nolimits_{j=1}^{\infty}}
(j^{2})^{\beta}a_{j}\sqrt{\frac{2}{\pi}}\sin jt
\end{multline*}
for $x(t)=%
{\displaystyle\sum\nolimits_{j=1}^{\infty}}
a_{j}\sqrt{\frac{2}{\pi}}\sin jt\in D((-\Delta)^{\beta})$ (\footnote{The
series is meant in $L^{2}$ but from the Carleson theorem it follows that
$x(t)=%
{\displaystyle\sum\nolimits_{j=1}^{\infty}}
a_{j}\sqrt{\frac{2}{\pi}}\sin jt$ a.e. on $(0,\pi)$ (cf. \cite[Theorem
5.17]{Brezis}).}).

Equality (\ref{ow}) and the fact that isolated points of the spectrum of a
self-adjoint operator are the eigenvalues imply that
\[
\sigma((-\Delta)^{\beta})=\sigma_{p}((-\Delta)^{\beta})=\{(j^{2})^{\beta
};\ j\in\mathbb{N}\}.
\]
The corresponding eigenspaces for $(-\Delta)$ and $(-\Delta)^{\beta}$ are the
same (it follows from a general result concerning the power of any
self-adjoint operator).

The operator $(-\Delta)^{\beta}$ will be called \textit{the Dirichlet-Laplace
operator of order }$\beta$, and the function $(-\Delta)^{\beta}x$ -
\textit{the Dirichlet-Laplacian of order }$\beta$\textit{ of} $x$.

We also have

\begin{lemma}
$D((-\Delta)^{\beta})$ with the scalar product%
\[
\left\langle x,y\right\rangle _{\beta}=\left\langle x,y\right\rangle _{L^{2}%
}+\left\langle (-\Delta)^{\beta}x,(-\Delta)^{\beta}y\right\rangle _{L^{2}}%
\]
is the Hilbert space.
\end{lemma}

\begin{proof}
The assertion follows from the fact that the operator $(-\Delta)^{\beta}$
being self-adjoint (cf. (\ref{samosp})) is closed.
\end{proof}

The scalar product $\left\langle \cdot,\cdot\right\rangle _{\beta}$ and the
scalar product%
\[
\left\langle x,y\right\rangle _{\thicksim\beta}=\left\langle (-\Delta)^{\beta
}x,(-\Delta)^{\beta}y\right\rangle _{L^{2}}%
\]
generate equivalent norms in $D((-\Delta)^{\beta})$. Indeed, it is sufficient
to observe that the following Poincare inequality holds true:
\begin{equation}
\left\Vert x\right\Vert _{L^{2}}^{2}=%
{\displaystyle\sum\nolimits_{j=1}^{\infty}}
a_{j}^{2}\leq%
{\displaystyle\sum\nolimits_{j=1}^{\infty}}
((j^{2})^{\beta})^{2}a_{j}^{2}=\left\Vert (-\Delta)^{\beta}x\right\Vert
_{L^{2}}^{2}=\left\Vert x\right\Vert _{\thicksim\beta}^{2} \label{Poincare}%
\end{equation}
for any $x(t)=%
{\displaystyle\sum\nolimits_{j=1}^{\infty}}
a_{j}\sqrt{\frac{2}{\pi}}\sin jt\in D((-\Delta)^{\beta})$. In the next, we
shall consider $D((-\Delta)^{\beta})$ with the norm $\left\Vert \cdot
\right\Vert _{\thicksim\beta}$.

\subsection{Embeddings}

From the description of the domain $D((-\Delta)^{\beta})$ it follows that%
\begin{equation}
D((-\Delta)^{\beta_{2}})\subset D((-\Delta)^{\beta_{1}}) \label{b1b2}%
\end{equation}
for any $0<\beta_{1}<\beta_{2}$. Using this relation and equality
(\ref{formulann}) with $A=(-\Delta)$ we assert that%
\[
C_{c}^{\infty}\subset D((-\Delta)^{\beta})
\]
for any $\beta>0$ ($C_{c}^{\infty}=C_{c}^{\infty}((0,\pi),\mathbb{R}^{m})$ is
the set of smooth functions with the supports contained in $(0,\pi)$).

We also have the following three lemmas.

\begin{lemma}
\label{zanurzenie2}If $\beta>\frac{1}{4}$, then%
\[
D((-\Delta)^{\beta})\subset L_{m}^{\infty}=L^{\infty}((0,\pi),\mathbb{R}^{m})
\]
and this embedding is continuous, more precisely,%
\[
\left\Vert x\right\Vert _{L_{m}^{\infty}}\leq\sqrt{\frac{2}{\pi}\zeta(4\beta
)}\left\Vert x\right\Vert _{\thicksim\beta}%
\]
for $x\in D((-\Delta)^{\beta})$, where $\zeta(4\beta)$ is the value of the
Riemann zeta function $\zeta(\gamma)=%
{\displaystyle\sum\nolimits_{j=1}^{\infty}}
\frac{1}{j^{\gamma}\text{ }}$ at $\gamma=4\beta$.
\end{lemma}

\begin{proof}
Let $x(t)=%
{\displaystyle\sum\nolimits_{j=1}^{\infty}}
a_{j}\sqrt{\frac{2}{\pi}}\sin jt\in D((-\Delta)^{\beta})$. Since
\[%
{\displaystyle\sum\nolimits_{j=1}^{\infty}}
((j^{2})^{\beta})^{2}a_{j}^{2}<\infty
\]
and $\beta>\frac{1}{4}$, therefore, for $t\in(0,\pi)$ a.e., we have%
\begin{multline*}
\left\vert x(t)\right\vert ^{2}=\left\vert
{\displaystyle\sum\nolimits_{j=1}^{\infty}}
a_{j}\sqrt{\frac{2}{\pi}}\sin jt\right\vert ^{2}\leq\frac{2}{\pi}\left(
{\displaystyle\sum\nolimits_{j=1}^{\infty}}
\left\vert a_{j}\right\vert \right)  ^{2}\\
=\frac{2}{\pi}\left(
{\displaystyle\sum\nolimits_{j=1}^{\infty}}
\frac{(j^{2})^{\beta}\left\vert a_{j}\right\vert }{(j^{2})^{\beta}}\right)
^{2}\leq\frac{2}{\pi}\left(
{\displaystyle\sum\nolimits_{j=1}^{\infty}}
((j^{2})^{\beta})^{2}\left\vert a_{j}\right\vert ^{2}\right)  \left(
{\displaystyle\sum\nolimits_{j=1}^{\infty}}
\frac{1}{\left(  (j^{2})^{\beta}\right)  ^{2}}\right) \\
=\frac{2}{\pi}\left\Vert x\right\Vert _{\thicksim\beta}^{2}\zeta
(4\beta)<\infty
\end{multline*}
and the proof is completed.
\end{proof}

\begin{lemma}
\label{zanurzenie1}If $\beta\geq\frac{1}{2}$, then
\[
D((-\Delta)^{\beta})\subset H_{0}^{1}%
\]
and, consequently,%
\[
D((-\Delta)^{\beta})\subset C=C([0,\pi],\mathbb{R}^{m}).
\]

\end{lemma}

\begin{proof}
Of course (cf. (\ref{b1b2})), it is sufficient to show that $D((-\Delta
)^{\frac{1}{2}})\subset H_{0}^{1}$. Indeed, let $x(t)=%
{\displaystyle\sum\nolimits_{j=1}^{\infty}}
a_{j}\sqrt{\frac{2}{\pi}}\sin jt\in D((-\Delta)^{\frac{1}{2}})$ and consider
this series on the interval $[0,\pi]$. The sequence $(S_{n})$ of partial sums
converges in $L^{2}$ to $x$. From the convergence of the series $%
{\displaystyle\sum\nolimits_{j=1}^{\infty}}
j^{2}a_{j}^{2}$ it follows that the sequence $(S_{n}^{\prime})$ of derivatives
converges in $L^{2}$ to a function. So (cf. \cite{Brezis}), one can choose a
subsequence $(S_{n_{k}}^{\prime})$ convergent a.e. on $[0,\pi]$ to this
function and bounded pointwise a.e. on $[0,\pi]$ by a function $g\in L^{2}$.
Consequently, the sequence $(S_{n_{k}}^{\prime})$ is equiabsolutely integrable
on $[0,\pi]$. So, the sequence $(S_{n_{k}})$ is equiabsolutely continuous on
$[0,\pi]$. Of course, $S_{n_{k}}(0)=0$, thus%
\[
\left\vert S_{n_{k}}(t)\right\vert =\left\vert S_{n_{k}}(0)+%
{\displaystyle\int\nolimits_{0}^{t}}
S_{n_{k}}^{\prime}(s)ds\right\vert \leq%
{\displaystyle\int\nolimits_{0}^{\pi}}
g(s)ds<\infty
\]
for $t\in\lbrack0,\pi]$. It means that elements of the sequence $(S_{n_{k}})$
satisfy the assumptions of the Ascoli-Arzela theorem for absolutely continuous
functions and, in consequence, there exists a subsequence $(S_{n_{k_{i}}})$
converging uniformly on $[0,\pi]$ to an absolutely continuous function
$\overline{x}$. Clearly, $(S_{n_{k_{i}}})$ converges to $\overline{x}$ in
$L^{2}$. The uniqueness of the limit in $L^{2}$ means that $x=\overline{x}$
a.e. on $(0,\pi)$. So, $x$ has a representative which is absolutely continuous
on $[0,\pi]$ and satisfies Dirichlet boundary conditions, i.e. $x\in
W_{0}^{1,1}((0,\pi),\mathbb{R}^{m})$ (the classical Sobolev space).
Consequently, there exists a function $g\in L^{1}$ such that%
\[%
{\displaystyle\int\nolimits_{0}^{\pi}}
x(t)\varphi^{\prime}(t)dt=-%
{\displaystyle\int\nolimits_{0}^{\pi}}
g(t)\varphi(t)dt
\]
for any $\varphi\in C_{c}^{\infty}$. But%
\begin{multline*}%
{\displaystyle\int\nolimits_{0}^{\pi}}
x(t)\varphi^{\prime}(t)dt=%
{\displaystyle\int\nolimits_{0}^{\pi}}
(%
{\displaystyle\sum\nolimits_{j=1}^{\infty}}
a_{j}\sqrt{\frac{2}{\pi}}\sin jt)\varphi^{\prime}(t)dt=%
{\displaystyle\int\nolimits_{0}^{\pi}}
\underset{n\rightarrow\infty}{\lim}S_{n}(t)\varphi^{\prime}(t)dt\\
=%
{\displaystyle\sum\nolimits_{j=1}^{\infty}}
{\displaystyle\int\nolimits_{0}^{\pi}}
a_{j}\sqrt{\frac{2}{\pi}}\sin jt\varphi^{\prime}(t)dt=%
{\displaystyle\sum\nolimits_{j=1}^{\infty}}
-%
{\displaystyle\int\nolimits_{0}^{\pi}}
ja_{j}\sqrt{\frac{2}{\pi}}\cos jt\varphi(t)dt\\
=-%
{\displaystyle\int\nolimits_{0}^{\pi}}
{\displaystyle\sum\nolimits_{j=1}^{\infty}}
ja_{j}\sqrt{\frac{2}{\pi}}\cos jt\varphi(t)dt
\end{multline*}
for $\varphi\in C_{c}^{\infty}$ (the last equality follows from the fact that
$%
{\displaystyle\sum\nolimits_{j=1}^{\infty}}
j^{2}a_{j}^{2}<\infty$ and, consequently, $%
{\displaystyle\sum\nolimits_{n=1}^{\infty}}
ja_{j}\sqrt{\frac{2}{\pi}}\cos jt\in L^{2}$). Thus,%
\[
g(t)=%
{\displaystyle\sum\nolimits_{n=1}^{\infty}}
ja_{j}\sqrt{\frac{2}{\pi}}\cos jt\in L^{2}%
\]
and, finally, $x\in H_{0}^{1}$.

The second part of the theorem follows from known property of Sobolev space
$W^{1,1}((0,\pi),\mathbb{R}^{m})$.
\end{proof}

\begin{lemma}
\label{equi}If $\beta>\frac{3}{4}$, then any bounded the set $B\subset
D((-\Delta)^{\beta})$ is equicontinuous on $[0,\pi]$.
\end{lemma}

\begin{proof}
Similarly as in the proof of Lemma \ref{zanurzenie2} we obtain%
\begin{multline*}
\left\vert x(t_{1})-x(t_{2})\right\vert ^{2}=\left\vert
{\displaystyle\sum\nolimits_{j=1}^{\infty}}
a_{j}\sqrt{\frac{2}{\pi}}(\sin jt_{1}-\sin jt_{2})\right\vert ^{2}\\
\leq(%
{\displaystyle\sum\nolimits_{j=1}^{\infty}}
\left\vert a_{j}\right\vert \sqrt{\frac{2}{\pi}}2\left\vert \sin\frac
{j(t_{1}-t_{2})}{2}\right\vert )^{2}\leq\frac{2}{\pi}\left\vert t_{1}%
-t_{2}\right\vert ^{2}(%
{\displaystyle\sum\nolimits_{j=1}^{\infty}}
\left\vert a_{j}\right\vert j)^{2}\\
=\frac{2}{\pi}\left\vert t_{1}-t_{2}\right\vert ^{2}\left(
{\displaystyle\sum\nolimits_{j=1}^{\infty}}
\frac{(j^{2})^{\beta}\left\vert a_{j}\right\vert }{(j^{2})^{\beta}j^{-1}%
}\right)  ^{2}\\
\leq\frac{2}{\pi}\left\vert t_{1}-t_{2}\right\vert ^{2}\left(
{\displaystyle\sum\nolimits_{j=1}^{\infty}}
((j^{2})^{\beta})^{2}\left\vert a_{j}\right\vert ^{2}\right)  \left(
{\displaystyle\sum\nolimits_{j=1}^{\infty}}
\frac{1}{(j^{2\beta-1})^{2}}\right) \\
=\frac{2}{\pi}\left\vert t_{1}-t_{2}\right\vert ^{2}\left\Vert x\right\Vert
_{\thicksim\beta}^{2}\zeta(4\beta-2)<\infty
\end{multline*}
for $t_{1},t_{2}\in(0,\pi)$ a.e., where $x(t)=%
{\displaystyle\sum\nolimits_{j=1}^{\infty}}
a_{j}\sqrt{\frac{2}{\pi}}\sin jt\in D((-\Delta)^{\beta})$. Identifying $x$
with its absolutely continuous representative on $[0,\pi]$ we assert that the
above estimation holds true for all $t_{1},t_{2}\in\lbrack0,\pi]$.
\end{proof}

Using Lemmas \ref{zanurzenie2}, \ref{zanurzenie1}, \ref{equi} we obtain

\begin{corollary}
If $\beta>\frac{3}{4}$, then the embedding%
\[
D((-\Delta)^{\beta})\subset C
\]
is compact.
\end{corollary}

\subsection{Equivalence of equations}

Fact that the operator $(-\Delta)^{\beta}$ ($\beta>0$) is self-adjoint means
that its domain satisfies the equality%
\begin{gather}
D((-\Delta)^{\beta})=\{x\in L^{2};\ \text{\textit{there exists} }z\in
L^{2}\text{ \textit{such that}}\label{sa11}\\%
{\displaystyle\int\nolimits_{0}^{\pi}}
x(t)(-\Delta)^{\beta}y(t)dt=%
{\displaystyle\int\nolimits_{0}^{\pi}}
z(t)y(t)dt\text{ \textit{for any} }y\in D((-\Delta)^{\beta})\}\nonumber
\end{gather}
and%
\begin{equation}
(-\Delta)^{\beta}x=z\label{sa2}%
\end{equation}
for $x\in D((-\Delta)^{\beta})$.

From (\ref{beta1beta2}) it follows that $x\in D((-\Delta)^{\beta})$ if and
only if $x\in D((-\Delta)^{\frac{\beta}{2}})$, $(-\Delta)^{\frac{\beta}{2}%
}x\in D((-\Delta)^{\frac{\beta}{2}})$ and, in such a case,%
\begin{equation}
(-\Delta)^{\frac{\beta}{2}}((-\Delta)^{\frac{\beta}{2}}x)=(-\Delta)^{\beta
}x\text{.} \label{betabeta2beta}%
\end{equation}
Using this fact and (\ref{sa11}), (\ref{sa2}), we obtain

\begin{lemma}
\label{slabe}If $\beta>0$ and $g\in L^{2}$, then $x\in D((-\Delta)^{\beta})$
and%
\[
(-\Delta)^{\beta}x=g
\]
if and only if $x\in D((-\Delta)^{\frac{\beta}{2}})$ and%
\[%
{\displaystyle\int\nolimits_{0}^{\pi}}
(-\Delta)^{\frac{\beta}{2}}x(t)(-\Delta)^{\frac{\beta}{2}}y(t)dt=%
{\displaystyle\int\nolimits_{0}^{\pi}}
g(t)y(t)dt
\]
for any $y\in D((-\Delta)^{\frac{\beta}{2}})$.
\end{lemma}

\section{Global implicit function theorem}

Let $X$ be a real Banach space and $I:X\rightarrow\mathbb{R}$ - a functional
of class $C^{1}$. We say that $I$ satisfies \textit{Palais-Smale (PS)
condition} if any sequence $(x_{k})$ such that

\begin{itemize}
\item[$\cdot$] $\left\vert I(x_{k})\right\vert \leq M$ for all $k\in
\mathbb{N}$ and some $M>0,$

\item[$\cdot$] $I^{\prime}(x_{k})\longrightarrow0,$
\end{itemize}

\noindent admits a convergent subsequence ($I^{\prime}(x_{k})$ denotes the
Frechet differential of $I$ at $x_{k}$). A sequence $(x_{k})$ satisfying the
above conditions is called the (PS) sequence for $I$.

We have (cf. \cite{Id}, \cite{Id2})

\begin{theorem}
\label{gift2}Let $X$, $U$ be real Banach spaces, $H$ - a real Hilbert space.
If $F:X\times U\rightarrow H$ is continuously differentiable with respect to
$(x,u)\in X\times U$ and

\begin{itemize}
\item[$\cdot$] for any $u\in U$, the functional
\[
\varphi:X\ni x\longmapsto\frac{1}{2}\left\Vert F(x,u)\right\Vert ^{2}%
\in\mathbb{R}%
\]
satisfies (PS) condition

\item[$\cdot$] $F_{x}^{\prime}(x,u):X\rightarrow H$ is bijective for any
$(x,u)\in X\times U$,
\end{itemize}

\noindent then there exists a unique function $\lambda:U\rightarrow X$ such
that $F(\lambda(u),u)=0$ for any $u\in U$ and this function is of class
$C^{1}$ with differential $\lambda^{\prime}(u)$ at $u$ given by
\begin{equation}
\lambda^{\prime}(u)=-[F_{x}^{\prime}(\lambda(u),u)]^{-1}\circ F_{u}^{\prime
}(\lambda(u),u). \label{wzor}%
\end{equation}

\end{theorem}

\section{A boundary value problem}

Let us consider boundary value problem (\ref{W}). Using the global implicit
function theorem, we shall show that (under suitable assumptions) this problem
has a unique solution $x_{u}\in D((-\Delta)^{\beta})$ corresponding to any
fixed $u\in L_{r}^{\infty}=L^{\infty}((0,\pi),\mathbb{R}^{r})$ and the
mapping
\[
L_{r}^{\infty}\ni u\longmapsto x_{u}\in D((-\Delta)^{\beta})
\]
is continuously differentiable.

Consider the mapping%
\[
F:D((-\Delta)^{\beta})\times L_{r}^{\infty}\ni(x,u)\longmapsto(-\Delta
)^{\beta}x(t)-f(t,x(t),u(t))\in L^{2}.
\]
We shall formulate conditions guaranteeing that

\begin{itemize}
\item[$\cdot$] $F$ is of class $C^{1}$

\item[$\cdot$] differential $F_{x}(x,u):D((-\Delta)^{\beta})\rightarrow L^{2}$
is bijective for any $(x,u)\in D((-\Delta)^{\beta})\times L_{r}^{\infty}$

\item[$\cdot$] for any $u\in L_{r}^{\infty}$, functional%
\[
\varphi:D((-\Delta)^{\beta})\ni x\longmapsto\frac{1}{2}\left\Vert
F(x,u)\right\Vert _{L^{2}}^{2}\in\mathbb{R}%
\]
satisfies (PS) condition.
\end{itemize}

\subsection{Smoothness of $F$}

Assume that function $f$ is measurable in $t\in(0,\pi)$, continuously
differentiable in $(x,u)\in\mathbb{R}^{m}\times\mathbb{R}^{r}$ and%
\begin{equation}
\left\vert f(t,x,u)\right\vert ,\left\vert f_{x}(t,x,u)\right\vert ,\left\vert
f_{u}(t,x,u)\right\vert \leq a(t)\gamma(\left\vert x\right\vert )+b(t)\delta
(\left\vert u\right\vert ) \label{fwzrost1}%
\end{equation}
for $(t,x,u)\in(0,\pi)\times\mathbb{R}^{m}\times\mathbb{R}^{r}$, where $a,b\in
L^{2}$ and $\gamma,\delta:\mathbb{R}_{0}^{+}\rightarrow\mathbb{R}_{0}^{+}$ are
continuous functions. We have

\begin{proposition}
\label{Diff}If $\beta>\frac{1}{4}$, then $F$ is of class $C^{1}$ and the
differential $F^{\prime}(x,u):D((-\Delta)^{\beta})\times L_{r}^{\infty
}\rightarrow L^{2}$ of $F$ at $(x,u)$ is given by%
\[
F^{\prime}(x,u)(h,v)=(-\Delta)^{\beta}h(t)-f_{x}(t,x(t),u(t))h(t)-f_{u}%
(t,x(t),u(t))v(t)
\]
for $(h,v)\in D((-\Delta)^{\beta})\times L_{r}^{\infty}$.
\end{proposition}

\begin{proof}
Smoothness of the first term of $F$ is obvious. So, let us consider the
mapping%
\[
G:D((-\Delta)^{\beta})\times L_{r}^{\infty}\ni(x,u)\longmapsto
f(t,x(t),u(t))\in L^{2}.
\]
We shall show that the mappings%
\[
G_{x}(x,u):D((-\Delta)^{\beta})\ni h\longmapsto f_{x}(t,x(t),u(t))h(t)\in
L^{2},
\]%
\[
G_{u}(x,u):L_{r}^{\infty}\ni v\longmapsto f_{u}(t,x(t),u(t))v(t)\in L^{2}%
\]
are partial Frechet differentials of $G$ at $(x,u)$ and mappings%
\begin{equation}
D((-\Delta)^{\beta})\times L_{r}^{\infty}\ni(x,u)\longmapsto G_{x}%
(x,u)\in\mathcal{L}(D((-\Delta)^{\beta}),L^{2}), \label{I}%
\end{equation}%
\begin{equation}
D((-\Delta)^{\beta})\times L_{r}^{\infty}\ni(x,u)\longmapsto G_{u}%
(x,u)\in\mathcal{L}(L_{r}^{\infty},L^{2}) \label{II}%
\end{equation}
are continuous. Of course, it is sufficient to check the differentiability in
Gateaux sense because continuity of the above two mappings implies that the
Gateaux differentials are Frechet ones.

So, let us consider differentiability of $G$ with respect to $x$. Linearity
and continuity of the mapping $G_{x}(x,u)$ are obvious (in view of Lemma
\ref{zanurzenie2}). To prove that $G_{x}(x,u)$ is Gateaux differential of $G$
with respect to $x$, we shall show that%
\begin{multline*}
\left\Vert \frac{G(x+\lambda_{k}h,u)-G(x,u)}{\lambda_{k}}-G_{x}%
(x,u)h\right\Vert _{L^{2}}^{2}\\
=%
{\displaystyle\int\nolimits_{0}^{\pi}}
\left\vert \frac{f(t,x(t)+\lambda_{k}h(t),u(t))-f(t,x(t),u(t))}{\lambda_{k}%
}-f_{x}(t,x(t),u(t))h(t)\right\vert ^{2}dt\\
\rightarrow0
\end{multline*}
for any sequence $(\lambda_{k})\subset(-1,1)$ such that $\lambda
_{k}\rightarrow0$. Indeed, the sequence of functions%
\[
t\longmapsto\frac{f(t,x(t)+\lambda_{k}h(t),u(t))-f(t,x(t),u(t))}{\lambda_{k}%
}-f_{x}(t,x(t),u(t))h(t)
\]
converges pointwise a.e. on $(0,\pi)$ to the zero function (by
differentiability of $f$ in $x$). Moreover, from the mean value theorem it
follows that this sequence is bounded by a function from $L^{2}$:%
\begin{multline*}
\left\vert \frac{f(t,x(t)+\lambda_{k}h(t),u(t))-f(t,x(t),u(t))}{\lambda_{k}%
}-f_{x}(t,x(t),u(t))h(t)\right\vert \\
=\left\vert f_{x}(t,x(t)+s_{t,k}\lambda_{k}h(t),u(t))h(t)-f_{x}%
(t,x(t),u(t))h(t)\right\vert \\
\leq const_{x,u,h}(a(t)+b(t))\left\vert h(t)\right\vert
\end{multline*}
where $s_{t,k}\in(0,1)$ and $const_{x,u,h}$ is a constant depending on
$x,u,h$. Thus, using the Lebesgue dominated convergence theorem we assert that
$G_{x}(x,u)$ is Gateaux differential of $G$ with respect to $x$.

In the same way, we check that $G_{u}(x,u)$ is Gateaux differential of $G$
with respect to $u$.

To finish the proof we shall show that the mappings (\ref{I}), (\ref{II}) are
continuous. Really, let $(x_{k},u_{k})\rightarrow(x_{0},u_{0})$ in
$D((-\Delta)^{\beta})\times L_{r}^{\infty}$. Then%
\begin{multline*}
\left\Vert (G_{x}(x_{k},u_{k})-G_{x}(x_{0},u_{0}))h\right\Vert _{L^{2}}^{2}\\
\leq%
{\displaystyle\int\nolimits_{0}^{\pi}}
\left\vert f_{x}(t,x_{k}(t),u_{k}(t))-f_{x}(t,x_{0}(t),u_{0}(t))\right\vert
^{2}\left\vert h(t)\right\vert ^{2}dt\\
\leq\left\Vert h\right\Vert _{\infty}^{2}%
{\displaystyle\int\nolimits_{0}^{\pi}}
\left\vert f_{x}(t,x_{k}(t),u_{k}(t))-f_{x}(t,x_{0}(t),u_{0}(t))\right\vert
^{2}dt\\
\leq\frac{2}{\pi}\zeta(4\beta)\left\Vert h\right\Vert _{\thicksim\beta}^{2}%
{\displaystyle\int\nolimits_{0}^{\pi}}
\left\vert f_{x}(t,x_{k}(t),u_{k}(t))-f_{x}(t,x_{0}(t),u_{0}(t))\right\vert
^{2}dt.
\end{multline*}
Consequently,%
\begin{multline*}
\left\Vert G_{x}(x_{k},u_{k})-G_{x}(x_{0},u_{0})\right\Vert _{\mathcal{L}%
(D((-\Delta)^{\beta}),L^{2})}\\
\leq\sqrt{\frac{2}{\pi}\zeta(4\beta)}(%
{\displaystyle\int\nolimits_{0}^{\pi}}
\left\vert f_{x}(t,x_{k}(t),u_{k}(t))-f_{x}(t,x_{0}(t),u_{0}(t))\right\vert
^{2}dt)^{\frac{1}{2}}.
\end{multline*}
Using Lemma \ref{zanurzenie2}, assumption (\ref{fwzrost1}) and the Lebesgue
dominated convergence theorem we assert that $G_{x}(x_{k},u_{k})\rightarrow
G_{x}(x_{0},u_{0})$ in $\mathcal{L}(D((-\Delta)^{\beta}),L^{2})$.

In a similar way, we check the continuity of the mapping%
\[
D((-\Delta)^{\beta})\times L_{r}^{\infty}\ni(x,u)\longmapsto G_{u}%
(x,u)\in\mathcal{L}(L_{r}^{\infty},L^{2}).
\]

The proof is completed.
\end{proof}

\subsection{Bijectivity of $F_{x}(x,u)$}

In view of the previous theorem and its proof, it is clear that if
$\beta>\frac{1}{4}$ and functions $f$, $f_{x}$ satisfy growth condition
(\ref{fwzrost1}), then the partial differential of $F$ with respect to $x$ is
of the form%
\[
F_{x}(x,u):D((-\Delta)^{\beta})\ni h\longmapsto(-\Delta)^{\beta}%
h(t)-f_{x}(t,x(t),u(t))h(t)\in L^{2}.
\]
for any $(x,u)\in D((-\Delta)^{\beta})\times L_{r}^{\infty}$. We also have

\begin{proposition}
\label{bijectivity}Assume that functions $f$, $f_{x}$ satisfy growth condition
(\ref{fwzrost1}). If $\beta>\frac{1}{2}$ and one of the following conditions
is satisfied

\begin{itemize}
\item[a)] $\left\Vert \Lambda\right\Vert _{L_{m\times m}^{1}}<\frac{\pi
}{2\zeta(2\beta)}$

\item[b)] $\Lambda(t)\leq0$, i.e. matrix $\Lambda(t)$ is nonpositive, for
$t\in(0,\pi)$ a.e.

\item[c)] $\Lambda\in L_{m\times m}^{\infty}$ and $\left\Vert \Lambda
\right\Vert _{\infty}<1$,
\end{itemize}

\noindent where $\Lambda(t):=f_{x}(t,x(t),u(t))$, $L_{m\times m}^{p}%
=L^{p}((0,\pi),\mathbb{R}^{m\times m})$ for $p=1,\infty$, then differential
$F_{x}(x,u):D((-\Delta)^{\beta})\rightarrow L^{2}$ is bijective ( \footnote{By
the norm of a matrix $C=[c_{i,j}]\in\mathbb{R}^{m\times m}$ we mean the value
$(%
{\displaystyle\sum\nolimits_{i,j=1}^{m}}
\left\vert c_{i,j}\right\vert ^{2})^{\frac{1}{2}}$.}).
\end{proposition}

\begin{remark}
In Part c) one can assume that $\beta>\frac{1}{4}$. In such a case the proof
of coercivity of $a$ remains unchanged and to show its continuity one
estimates%
\begin{align*}
\left\vert a(h,y)\right\vert  &  \leq\left\Vert h\right\Vert _{\thicksim
\frac{\beta}{2}}\left\Vert y\right\Vert _{\thicksim\frac{\beta}{2}}+\left\Vert
\Lambda\right\Vert _{\infty}\left\Vert h\right\Vert _{L^{2}}\left\Vert
y\right\Vert _{L^{2}}\\
&  \leq(1+\left\Vert \Lambda\right\Vert _{\infty})\left\Vert h\right\Vert
_{\thicksim\frac{\beta}{2}}\left\Vert y\right\Vert _{\thicksim\frac{\beta}{2}%
}.
\end{align*}

\end{remark}

\begin{proof}
[Proof of the Proposition]We shall show that, for any function $g\in L^{2}$,
equation%
\begin{equation}
(-\Delta)^{\beta}h(t)-\Lambda(t)h(t)=g(t) \label{liniowe}%
\end{equation}
has a unique solution in $D((-\Delta)^{\beta})$. Using Lemma \ref{slabe}, we
see that it is equivalent to show that there exists a unique function $h\in
D((-\Delta)^{\frac{\beta}{2}})$ such that%
\[%
{\displaystyle\int\nolimits_{0}^{\pi}}
(-\Delta)^{\frac{\beta}{2}}h(t)(-\Delta)^{\frac{\beta}{2}}y(t)dt=%
{\displaystyle\int\nolimits_{0}^{\pi}}
(\Lambda(t)h(t)+g(t))y(t)dt
\]
for any $y\in D((-\Delta)^{\frac{\beta}{2}})$. So, let us define a bilinear
form $a:D((-\Delta)^{\frac{\beta}{2}})\times D((-\Delta)^{\frac{\beta}{2}%
})\rightarrow\mathbb{R}$ by%
\[
a(h,y)=%
{\displaystyle\int\nolimits_{0}^{\pi}}
(-\Delta)^{\frac{\beta}{2}}h(t)(-\Delta)^{\frac{\beta}{2}}y(t)dt-%
{\displaystyle\int\nolimits_{0}^{\pi}}
\Lambda(t)h(t)y(t)dt\text{.}%
\]
This form is continuous. Indeed (cf. Lemma \ref{zanurzenie2}),%
\begin{align*}
\left\vert a(h,y)\right\vert  &  \leq\left\Vert h\right\Vert _{\thicksim
\frac{\beta}{2}}\left\Vert y\right\Vert _{\thicksim\frac{\beta}{2}}+\left\Vert
\Lambda\right\Vert _{L^{1}}\left\Vert h\right\Vert _{\infty}\left\Vert
y\right\Vert _{\infty}\\
&  \leq(1+\left\Vert \Lambda\right\Vert _{L^{1}}\frac{2}{\pi}\varsigma
(2\beta))\left\Vert h\right\Vert _{\thicksim\frac{\beta}{2}}\left\Vert
y\right\Vert _{\thicksim\frac{\beta}{2}}%
\end{align*}
for $h,y\in D((-\Delta)^{\frac{\beta}{2}})$. Moreover,

\noindent\textbf{Part a.}%
\begin{align*}
\left\vert a(h,h)\right\vert  &  =\left\vert
{\displaystyle\int\nolimits_{0}^{\pi}}
(-\Delta)^{\frac{\beta}{2}}h(t)(-\Delta)^{\frac{\beta}{2}}h(t)dt-%
{\displaystyle\int\nolimits_{0}^{\pi}}
\Lambda(t)h(t)h(t)dt\right\vert \\
&  \geq\left\Vert h\right\Vert _{\thicksim\frac{\beta}{2}}^{2}-\left\Vert
\Lambda\right\Vert _{L^{1}}\left\Vert h\right\Vert _{\infty}^{2}%
\geq(1-\left\Vert \Lambda\right\Vert _{L^{1}}\frac{2}{\pi}\varsigma
(2\beta))\left\Vert h\right\Vert _{\thicksim\frac{\beta}{2}}^{2},
\end{align*}

\noindent\textbf{Part b. }%
\[
\left\vert a(h,h)\right\vert \geq%
{\displaystyle\int\nolimits_{0}^{\pi}}
(-\Delta)^{\frac{\beta}{2}}h(t)(-\Delta)^{\frac{\beta}{2}}h(t)dt-%
{\displaystyle\int\nolimits_{0}^{\pi}}
\Lambda(t)h(t)h(t)dt\geq\left\Vert h\right\Vert _{\thicksim\frac{\beta}{2}%
}^{2}.
\]

\noindent\textbf{Part c. }%
\begin{align*}
\left\vert a(h,h)\right\vert  &  =\left\vert
{\displaystyle\int\nolimits_{0}^{\pi}}
(-\Delta)^{\frac{\beta}{2}}h(t)(-\Delta)^{\frac{\beta}{2}}h(t)dt-%
{\displaystyle\int\nolimits_{0}^{\pi}}
\Lambda(t)h(t)h(t)dt\right\vert \\
&  \geq\left\Vert h\right\Vert _{\thicksim\frac{\beta}{2}}^{2}-\left\Vert
\Lambda\right\Vert _{\infty}\left\Vert h\right\Vert _{L^{2}}^{2}%
\geq(1-\left\Vert \Lambda\right\Vert _{\infty})\left\Vert h\right\Vert
_{\thicksim\frac{\beta}{2}}^{2}.
\end{align*}

So, $a$ is coercive. From Lax-Milgram theorem it follows that for any linear
continuous functional $l:D((-\Delta)^{\frac{\beta}{2}})\rightarrow\mathbb{R}$
there exists a unique $h\in D((-\Delta)^{\frac{\beta}{2}})$ such that%
\[
a(h,y)=l(y)
\]
for any $y\in D((-\Delta)^{\frac{\beta}{2}})$. Since the functional%
\[
D((-\Delta)^{\frac{\beta}{2}})\ni y\longmapsto%
{\displaystyle\int\nolimits_{0}^{\pi}}
g(t)y(t)dt\in\mathbb{R}%
\]
is linear and continuous, therefore there exists a unique $h\in D((-\Delta
)^{\frac{\beta}{2}})$ such that%
\[%
{\displaystyle\int\nolimits_{0}^{\pi}}
(-\Delta)^{\frac{\beta}{2}}h(t)(-\Delta)^{\frac{\beta}{2}}y(t)dt-%
{\displaystyle\int\nolimits_{0}^{\pi}}
\Lambda(t)h(t)y(t)dt=%
{\displaystyle\int\nolimits_{0}^{\pi}}
g(t)y(t)dt
\]
for any $y\in D((-\Delta)^{\frac{\beta}{2}})$. The proof is completed.
\end{proof}

\subsection{(PS) condition}

In the same way as in the proof of Proposition \ref{bijectivity} one can show
that, for any $\beta>0$ and any function $g\in L^{2}$, there exists a unique
function $x_{g}\in D((-\Delta)^{\frac{\beta}{2}})$ such that%
\[%
{\displaystyle\int\nolimits_{0}^{\pi}}
(-\Delta)^{\frac{\beta}{2}}x_{g}(t)(-\Delta)^{\frac{\beta}{2}}y(t)dt=%
{\displaystyle\int\nolimits_{0}^{\pi}}
g(t)y(t)dt
\]
for any $y\in D((-\Delta)^{\frac{\beta}{2}})$. It means, in view of Lemma
\ref{slabe}, that the following lemma holds true.

\begin{lemma}
For any $\beta>0$ and any $g\in L^{2}$, there exists a unique solution
$x_{g}\in D((-\Delta)^{\beta})$ of the equation%
\[
(-\Delta)^{\beta}x=g.
\]

\end{lemma}

Moreover, we have

\begin{lemma}
\label{compact}If $\beta>\frac{1}{2}$, then the operator
\[
\lbrack(-\Delta)^{\beta}]^{-1}:L^{2}\ni g\longmapsto x_{g}\in L^{2}%
\]
is compact, i.e. the image of any bounded set in $L^{2}$ is relatively compact
in $L^{2}$.
\end{lemma}

\begin{proof}
Since $x_{(g_{1},...,g_{m})}=(x_{g_{1}},...,x_{g_{m}})$ for any $(g_{1}%
,...,g_{m})\in L^{2}$, therefore one can assume that $m=1$.

Let us recall the Kolmogorov-Frechet-Riesz theorem (cf. \cite{Brezis}): if
$\mathcal{F}$ is a bounded set in $L^{p}(\mathbb{R}^{n})$ ($1\leq p<\infty$)
and%
\begin{equation}
\underset{\varepsilon>0}{\forall}\;\underset{\delta>0}{\exists}\;\underset
{\left\vert h\right\vert <\delta}{\forall}\;\underset{f\in\mathcal{F}}%
{\forall}\;\left\Vert \tau_{h}f-f\right\Vert _{L^{p}(\mathbb{R}^{n}%
)}<\varepsilon\label{KFR}%
\end{equation}
(here, $\tau_{h}f(x)=f(x+h)$), then $\mathcal{F\mid}_{\Omega}$ is relatively
compact in $L^{p}(\Omega)$ for any measurable set $\Omega\subset\mathbb{R}%
^{n}$ with finite Lebesgue measure.

Let $G\subset L^{2}((0,\pi),\mathbb{R})$ be a set bounded by a constant $C$ in
$L^{2}((0,\pi),\mathbb{R})$. Consider a function
\[
g(t)=%
{\displaystyle\sum\nolimits_{j=1}^{\infty}}
b_{j}^{g}\sqrt{\frac{2}{\pi}}\sin jt\in G
\]
and the function
\[
x_{g}(t)=%
{\displaystyle\sum\nolimits_{j=1}^{\infty}}
a_{j}^{g}\sqrt{\frac{2}{\pi}}\sin jt
\]
(both equalities and convergences are meant in $L^{2}$ and, in view of the
Carleson theorem, a.e. on $(0,\pi)$). Since
\[
(-\Delta)^{\beta}x_{g}(t)=g(t)
\]
i.e.%
\[%
{\displaystyle\sum\nolimits_{j=1}^{\infty}}
(j^{2})^{\beta}a_{j}^{g}\sqrt{\frac{2}{\pi}}\sin jt=%
{\displaystyle\sum\nolimits_{j=1}^{\infty}}
b_{j}^{g}\sqrt{\frac{2}{\pi}}\sin jt,
\]
therefore%
\[
a_{j}^{g}=\frac{b_{j}^{g}}{(j^{2})^{\beta}}%
\]
for $j\in\mathbb{N}$. Now, we shall show that the set of functions
$\{\widetilde{x}_{g};\ g\in G\}$ where
\[
\widetilde{x}_{g}:\mathbb{R}\ni t\longmapsto\left\{
\begin{array}
[c]{ccc}%
x_{g}(t) & ; & t\in(0,\pi)\\
0 & ; & \text{otherwise}%
\end{array}
\right.  ,
\]
satisfies condition (\ref{KFR}) (of course, it is bounded in $L^{2}%
(\mathbb{R},\mathbb{R})$). Let us fix $0<h<\pi$ and consider the integral%
\begin{align}
&
{\displaystyle\int\nolimits_{-\infty}^{\infty}}
\left\vert \widetilde{x}_{g}(t+h)-\widetilde{x}_{g}(t)\right\vert ^{2}dt=%
{\displaystyle\int\nolimits_{-h}^{0}}
\left\vert \widetilde{x}_{g}(t+h)\right\vert ^{2}dt\nonumber\\
&  +%
{\displaystyle\int\nolimits_{0}^{\pi-h}}
\left\vert \widetilde{x}_{g}(t+h)-\widetilde{x}_{g}(t)\right\vert ^{2}dt+%
{\displaystyle\int\nolimits_{\pi-h}^{\pi}}
\left\vert \widetilde{x}_{g}(t+h)-\widetilde{x}_{g}(t)\right\vert
^{2}dt\nonumber\\
&  =%
{\displaystyle\int\nolimits_{0}^{h}}
\left\vert x_{g}(t)\right\vert ^{2}dt+%
{\displaystyle\int\nolimits_{0}^{\pi-h}}
\left\vert x_{g}(t+h)-x_{g}(t)\right\vert ^{2}dt+%
{\displaystyle\int\nolimits_{\pi-h}^{\pi}}
\left\vert x_{g}(t)\right\vert ^{2}dt. \label{sumatrzy}%
\end{align}
The first term of the above sum can be estimated in the following way (to
obtain third inequality we use H\"{o}lder inequality for series)%
\begin{multline*}%
{\displaystyle\int\nolimits_{0}^{h}}
\left\vert x_{g}(t)\right\vert ^{2}dt=%
{\displaystyle\int\nolimits_{0}^{h}}
\left\vert
{\displaystyle\sum\nolimits_{j=1}^{\infty}}
\frac{b_{j}^{g}}{(j^{2})^{\beta}}\sqrt{\frac{2}{\pi}}\sin jt\right\vert
^{2}dt\\
\leq\frac{2}{\pi}%
{\displaystyle\int\nolimits_{0}^{h}}
(%
{\displaystyle\sum\nolimits_{j=1}^{\infty}}
\frac{\left\vert b_{j}^{g}\right\vert }{(j^{2})^{\beta}})^{2}dt\leq\frac
{2}{\pi}h%
{\displaystyle\sum\nolimits_{j=1}^{\infty}}
\left\vert b_{j}^{g}\right\vert ^{2}%
{\displaystyle\sum\nolimits_{j=1}^{\infty}}
\frac{1}{(j^{2})^{2\beta}}\\
=\frac{2}{\pi}h\left\Vert g\right\Vert _{L^{2}}^{2}\zeta(4\beta)\leq\frac
{2}{\pi}C\zeta(4\beta)h.
\end{multline*}
In the same way one can estimate third term of (\ref{sumatrzy}).

For the second term, we have%
\begin{multline*}%
{\displaystyle\int\nolimits_{0}^{\pi-h}}
\left\vert x_{g}(t+h)-x_{g}(t)\right\vert ^{2}dt\\
=%
{\displaystyle\int\nolimits_{0}^{\pi-h}}
\left\vert
{\displaystyle\sum\nolimits_{j=1}^{\infty}}
\frac{b_{j}^{g}}{(j^{2})^{\beta}}\sqrt{\frac{2}{\pi}}(\sin j(t+h)-\sin
jt)\right\vert ^{2}dt\\
\leq%
{\displaystyle\int\nolimits_{0}^{\pi-h}}
(%
{\displaystyle\sum\nolimits_{j=1}^{\infty}}
\frac{\left\vert b_{j}^{g}\right\vert }{(j^{2})^{\beta}}\sqrt{\frac{2}{\pi}%
}\left\vert 2\sin\frac{jh}{2}\cos(jt+\frac{jh}{2})\right\vert )^{2}dt\\
\leq\frac{8}{\pi}%
{\displaystyle\int\nolimits_{0}^{\pi-h}}
(%
{\displaystyle\sum\nolimits_{j=1}^{\infty}}
\frac{\left\vert b_{j}^{g}\right\vert }{(j^{2})^{\beta}}\left\vert \sin
\frac{jh}{2}\right\vert )^{2}dt\leq\frac{8}{\pi}(\pi-h)%
{\displaystyle\sum\nolimits_{j=1}^{\infty}}
\left\vert b_{j}^{g}\right\vert ^{2}%
{\displaystyle\sum\nolimits_{j=1}^{\infty}}
\frac{\sin^{2}\frac{jh}{2}}{(j^{2})^{2\beta}}\\
\leq\frac{8}{\pi}(\pi-h)C%
{\displaystyle\sum\nolimits_{j=1}^{\infty}}
\frac{jh}{2(j^{2})^{2\beta}}\leq4Ch%
{\displaystyle\sum\nolimits_{j=1}^{\infty}}
\frac{1}{j^{4\beta-1}}=4C\zeta(4\beta-1)h.
\end{multline*}
If $-\pi<h<0$, we proceed in the same way.

Finally,%
\[
\left\Vert \tau_{h}f-f\right\Vert _{L^{p}(\mathbb{R}^{n})}\leq const\left\vert
h\right\vert
\]
for $\left\vert h\right\vert <\pi$. So, the set $\{\left.  \widetilde{x}%
_{g}\right\vert _{(0,\pi)};\ g\in G\}=\{x_{g};\ g\in G\}$ is relatively
compact in $L^{2}$. The proof is completed.
\end{proof}

Using the above lemma we obtain

\begin{lemma}
\label{weak conver}If $\beta>\frac{1}{2}$ and $x_{k}\rightharpoonup x_{0}$
weakly in $D((-\Delta)^{\beta})$, then $x_{k}\rightarrow x_{0}$ strongly in
$L^{2}$ and $(-\Delta)^{\beta}x_{k}\rightharpoonup(-\Delta)^{\beta}x_{0}$
weakly in $L^{2}$.
\end{lemma}

\begin{proof}
From the continuity of the linear operators%
\[
D((-\Delta)^{\beta})\ni x\longmapsto x\in L^{2},
\]%
\[
D((-\Delta)^{\beta})\ni x\longmapsto(-\Delta)^{\beta}x\in L^{2}%
\]
it follows that $x_{k}\rightharpoonup x_{0}$ weakly in $L^{2}$ and
$(-\Delta)^{\beta}x_{k}\rightharpoonup(-\Delta)^{\beta}x_{0}$ weakly in
$L^{2}$. Lemma \ref{compact} implies that the sequence $(x_{k})$ contains a
subsequence $(x_{k_{i}})$ converging strongly in $L^{2}$ to a limit. Of
course, this limit is the function $x_{0}$, i.e. $x_{k_{i}}\rightarrow x_{0}$
strongly in $L^{2}$. Supposing contrary and repeating the above argumentation
we assert that $x_{k}\rightarrow x_{0}$ strongly in $L^{2}$.
\end{proof}

\begin{remark}
Lemmas \ref{compact} and Lemma \ref{weak conver} are valid for any $\beta>0$.
The proofs of such stronger results, in the case of bounded open set
$\Omega\subset\mathbb{R}^{n}$ ($n\geq1$), can be found in \cite{Id3}. We give
here weaker theorems for two reasons. First, to prove more general results (in
fact, a counterpart of Lemma \ref{compact} because the proof of Lemma
\ref{weak conver} remains unchanged) some additional considerations,
concerning the spectral representation of the inverse operator, are needed.
Second, due to the other assumptions (cf. Proposition \ref{bijectivity})
assumption $\beta>\frac{1}{2}$ in Theorem \ref{final} can not be omitted.
\end{remark}

The main tool that we shall use to prove that $\varphi$ satisfies (PS)
condition is the following lemma.

\begin{lemma}
If $\beta>\frac{1}{4}$, $f$ satisfies the growth condition%
\[
\left\vert f(t,x,u)\right\vert \leq a(t)\left\vert x\right\vert +b(t)\delta
(\left\vert u\right\vert )
\]
for $(t,x,u)\in(0,\pi)\times\mathbb{R}^{m}\times\mathbb{R}^{r}$, where $a,b\in
L^{2}$, $\delta:\mathbb{R}_{0}^{+}\rightarrow\mathbb{R}_{0}^{+}$ is a
continuous function and%
\begin{equation}
\sqrt{\frac{2}{\pi}\zeta(4\beta)}\left\Vert a\right\Vert _{L^{2}}<1,
\label{niekoer}%
\end{equation}
then, for any $u\in L_{r}^{\infty}$, the functional
\[
\varphi:D((-\Delta)^{\beta})\ni x\longmapsto\frac{1}{2}\left\Vert
F(x,u)\right\Vert _{L^{2}}^{2}\in\mathbb{R}%
\]
is coercive, i.e.%
\[
\left\Vert x\right\Vert _{\thicksim\beta}\rightarrow\infty\Longrightarrow
\varphi(x)\rightarrow\infty.
\]

\end{lemma}

\begin{proof}
We have%
\begin{align*}
\left\Vert F(x,u)\right\Vert _{L^{2}}  &  =\left\Vert (-\Delta)^{\beta
}x(t)-f(t,x(t),u(t))\right\Vert _{L^{2}}\\
&  \geq\left\Vert (-\Delta)^{\beta}x(t)\right\Vert _{L^{2}}-\left\Vert
f(t,x(t),u(t))\right\Vert _{L^{2}}.
\end{align*}
But%
\begin{multline*}
\left\Vert f(t,x(t),u(t))\right\Vert _{L^{2}}\leq(%
{\displaystyle\int\nolimits_{0}^{\pi}}
(a(t)\left\vert x(t)\right\vert +b(t)\delta(\left\vert u(t)\right\vert
))^{2}dt)^{\frac{1}{2}}\\
\leq(%
{\displaystyle\int\nolimits_{0}^{\pi}}
\left\vert a(t)\right\vert ^{2}\left\vert x(t)\right\vert ^{2}dt)^{\frac{1}%
{2}}+D\leq\left\Vert x\right\Vert _{\infty}\left\Vert a\right\Vert _{L^{2}%
}+D\\
\leq\sqrt{\frac{2}{\pi}\zeta(4\beta)}\left\Vert a\right\Vert _{L^{2}%
}\left\Vert (-\Delta)^{\beta}x\right\Vert _{L^{2}}+D
\end{multline*}
where $D=(%
{\displaystyle\int\nolimits_{0}^{\pi}}
\left\vert b(t)\right\vert ^{2}(\delta(\left\vert u(t)\right\vert
))^{2}dt)^{\frac{1}{2}}$. Thus,%
\begin{align*}
\left\Vert F(x,u)\right\Vert _{L^{2}}  &  \geq\left\Vert (-\Delta)^{\beta
}x\right\Vert _{L^{2}}-\sqrt{\frac{2}{\pi}\zeta(4\beta)}\left\Vert
a\right\Vert _{L^{2}}\left\Vert (-\Delta)^{\beta}x\right\Vert _{L^{2}}-D\\
&  =(1-\sqrt{\frac{2}{\pi}\zeta(4\beta)}\left\Vert a\right\Vert _{L^{2}%
})\left\Vert x\right\Vert _{\thicksim\beta}-D.
\end{align*}
It means that $\varphi$ is coercive.
\end{proof}

Now, we are in a position to prove that the functional $\varphi$ satisfies
(PS) condition. Namely, we have

\begin{proposition}
If $\beta>\frac{1}{2}$, $f$ and $f_{x}$ satisfy the growth conditions%
\[
\left\vert f(t,x,u)\right\vert \leq a(t)\left\vert x\right\vert +b(t)\delta
(\left\vert u\right\vert )
\]%
\[
\left\vert f_{x}(t,x,u)\right\vert \leq a(t)\gamma(\left\vert x\right\vert
)+b(t)\delta(\left\vert u\right\vert )
\]
for $(t,x,u)\in(0,\pi)\times\mathbb{R}^{m}\times\mathbb{R}^{r}$, where $a,b\in
L^{2}$ and $\gamma,\delta:\mathbb{R}_{0}^{+}\rightarrow\mathbb{R}_{0}^{+}$ are
continuous functions, and (\ref{niekoer}) holds true, then $\varphi$ (with any
fixed $u\in L_{r}^{\infty}$) satisfies (PS) condition.
\end{proposition}

\begin{proof}
From Proposition \ref{Diff} it follows that $\varphi$ is of class $C^{1}$ and
its differential $\varphi^{\prime}(x):D((-\Delta)^{\beta})\rightarrow
\mathbb{R}$ is given by%
\begin{multline*}
\varphi^{\prime}(x)h\\
=%
{\displaystyle\int\nolimits_{0}^{\pi}}
((-\Delta)^{\beta}x(t)-f(t,x(t),u(t)))((-\Delta)^{\beta}h(t)-f_{x}%
(t,x(t),u(t))h(t))dt
\end{multline*}
for $h\in D((-\Delta)^{\beta})$. Consequently, for $x_{k},x_{0}\in
D((-\Delta)^{\beta})$, we have%
\[
(\varphi^{\prime}(x_{k})-\varphi^{\prime}(x_{0}))(x_{k}-x_{0})=\left\Vert
x_{k}-x_{0}\right\Vert _{\thicksim\frac{\beta}{2}}^{2}+%
{\displaystyle\sum\nolimits_{i=1}^{5}}
\psi_{i}(x_{k})
\]
where%
\[
\psi_{1}(x_{k})=%
{\displaystyle\int\nolimits_{0}^{\pi}}
(-\Delta)^{\beta}x_{k}(t)f_{x}(t,x_{k}(t),u(t))(x_{0}(t)-x_{k}(t))dt,
\]%
\[
\psi_{2}(x_{k})=%
{\displaystyle\int\nolimits_{0}^{\pi}}
(-\Delta)^{\beta}x_{0}(t)f_{x}(t,x_{0}(t),u(t))(x_{k}(t)-x_{0}(t))dt,
\]%
\[
\psi_{3}(x_{k})=%
{\displaystyle\int\nolimits_{0}^{\pi}}
f(t,x_{k}(t),u(t))f_{x}(t,x_{k}(t),u(t))(x_{k}(t)-x_{0}(t))dt,
\]%
\[
\psi_{4}(x_{k})=%
{\displaystyle\int\nolimits_{0}^{\pi}}
f(t,x_{0}(t),u(t))f_{x}(t,x_{0}(t),u(t))(x_{0}(t)-x_{k}(t))dt,
\]%
\begin{multline*}
\psi_{5}(x_{k})=%
{\displaystyle\int\nolimits_{0}^{\pi}}
(f(t,x_{0}(t),u(t))-f(t,x_{k}(t),u(t)))\\
((-\Delta)^{\beta}x_{k}(t)-(-\Delta)^{\beta}x_{0}(t))dt.
\end{multline*}

Now, let $(x_{k})$ be a (PS) sequence for $\varphi$. Since $\varphi$ is
coercive, therefore $(x_{k})$ is bounded in $D((-\Delta)^{\beta})$. So, one
can choose a subsequence $(x_{k_{i}})$ weakly converging in $D((-\Delta
)^{\beta})$ to some $x_{0}$. From Lemma \ref{weak conver} it follows that
$x_{k_{i}}\rightarrow x_{0}$ strongly in $L^{2}$ and $(-\Delta)^{\beta
}x_{k_{i}}(t)\rightharpoonup(-\Delta)^{\beta}x_{0}(t)$ weakly in $L^{2}$.
Since the sequence $(x_{k_{i}})$ is bounded in $D((-\Delta)^{\beta})$,
therefore it is bounded in $L_{m}^{\infty}$ and, consequently ($\beta>\frac
{1}{2}$), in $C$. Moreover, there exists a subsequence of the sequence
$(x_{k_{i}})$ (let us denote it by $(x_{k_{i}})$) converging to $x_{0}$
pointwise a.e. on $(0,\pi)$.

Term $\psi_{1}(x_{k_{i}})$ tends to zero. Indeed, functions $f_{x}(t,x_{k_{i}%
}(t),u(t))$, $k\in\mathbb{N}$, are equibounded on $(0,\pi)$ by a squared
integrable function. Functions $f_{x}(t,x_{k_{i}}(t),u(t))(x_{0}(t)-x_{m_{k}%
}(t))$ belong to $L^{2}$ and converge pointwise (a.e. on $(0,\pi)$) to zero
function. Moreover, they are equibounded on $(0,\pi)$ by a squared integrable
function. So, from the Lebesgue dominated convergence theorem it follows that
the sequence
\[
(f_{x}(t,x_{k_{i}}(t),u(t))(x_{0}(t)-x_{m_{k}}(t)))
\]
converges in $L^{2}$ to the zero function. Thus, in view of the weak
convergence of the sequence $((-\Delta)^{\beta}x_{k})$ to $(-\Delta)^{\beta
}x_{0}$ in $L^{2}$, $\psi_{1}(x_{k_{i}})\rightarrow0$.

Similarly, $\psi_{l}(x_{k_{i}})\rightarrow0$ for remaining $l$.

Finally, since
\[
\varphi^{\prime}(x_{k_{i}})(x_{k_{i}}-x_{0})\rightarrow0,
\]%
\[
\varphi^{\prime}(x_{0})(x_{k_{i}}-x_{0})\rightarrow0,
\]
therefore%
\[
\left\Vert x_{k_{i}}-x_{0}\right\Vert _{\thicksim\frac{\beta}{2}}%
^{2}\rightarrow0,
\]
i.e. $\varphi$ satisfies (PS) condition.
\end{proof}

\section{Final result}

Thus, we have proved

\begin{theorem}
\label{final}Assume that $\beta>\frac{1}{2}$, function $f$ is measurable in
$t\in(0,\pi)$, continuously differentiable in $(x,u)\in\mathbb{R}^{m}%
\times\mathbb{R}^{r}$ and%
\[
\left\vert f(t,x,u)\right\vert \leq a(t)\left\vert x\right\vert +b(t)\delta
(\left\vert u\right\vert )
\]%
\[
\left\vert f_{x}(t,x,u)\right\vert ,\left\vert f_{u}(t,x,u)\right\vert \leq
a(t)\gamma(\left\vert x\right\vert )+b(t)\delta(\left\vert u\right\vert )
\]
for $(t,x,u)\in(0,\pi)\times\mathbb{R}^{m}\times\mathbb{R}^{r}$, where $a,b\in
L^{2}$, $\gamma,\delta:\mathbb{R}_{0}^{+}\rightarrow\mathbb{R}_{0}^{+}$ are
continuous functions and%
\[
\sqrt{\frac{2}{\pi}\zeta(4\beta)}\left\Vert a\right\Vert _{L^{2}}<1.
\]
If, for any pair $(x,u)\in D((-\Delta)^{\beta})\times L_{r}^{\infty}$ one of
the following assumptions is satisfied

\begin{itemize}
\item[a)] $\left\Vert f_{x}(t,x(t),u(t))\right\Vert _{L_{m\times m}^{1}}%
<\frac{\pi}{2\zeta(2\beta)}$

\item[b)] $f_{x}(t,x(t),u(t))\leq0$ for $t\in(0,\pi)$ a.e.

\item[c)] $f_{x}(t,x(t),u(t))\in L_{m\times m}^{\infty}$ and $\left\Vert
f_{x}(t,x(t),u(t))\right\Vert _{\infty}<1$,
\end{itemize}

\noindent then, for any $u\in L_{r}^{\infty}$, there exists a unique solution
$x_{u}\in D((-\Delta)^{\beta})$ of problem (\ref{W}) and the mapping%
\[
\lambda:L_{r}^{\infty}\ni u\longmapsto x_{u}\in D((-\Delta)^{\beta})
\]
is continuously differentiable with the differential $\lambda^{\prime}(u)$ at
$u\in L_{r}^{\infty}$ such that, for any $v\in L_{r}^{\infty}$,%
\[
(-\Delta)^{\beta}(\lambda^{\prime}(u)v)(t)-f_{x}(t,x_{u}(t),u(t))(\lambda
^{\prime}(u)v)(t)=f_{u}(t,x_{u}(t),u(t))v(t)
\]
for $t\in(0,\pi)$ a.e.
\end{theorem}

\begin{remark}
Thus, for any $u\in L_{r}^{\infty}$, $v\in D((-\Delta)^{\beta})$ the function
$\lambda^{\prime}(u)v\in D((-\Delta)^{\beta})$ is a solution to the equation%
\[
(-\Delta)^{\beta}y(t)-f_{x}(t,x_{u}(t),u(t))y(t)=f_{u}(t,x_{u}%
(t),u(t))v(t),\ t\in(0,\pi)\text{ a.e.}%
\]

\end{remark}

\begin{example}
Let $\beta>\frac{1}{2},\ m=2,\ r=2$. It is easy to see that the function
\begin{align*}
f(t,x,u)  &  =(f^{1}(t,x_{1},x_{2},u_{1},u_{2}),f^{2}(t,x_{1},x_{2}%
,u_{1},u_{2}))\\
&  =(a\sin(x_{2})+t^{-\frac{1}{3}}e^{u_{1}},b\cos(x_{1})+tu_{2})
\end{align*}
satisfies assumptions of Theorem \ref{final} (including a)) with
\[
a(t)=\sqrt{a^{2}+b^{2}},\ \gamma(s)=\sqrt{2},\ b(t)=t^{-\frac{1}{3}%
}+t+\left\vert b\right\vert ,\ \delta(s)=e^{s},
\]
where $a,b\in\mathbb{R}$ are such that
\[
\sqrt{a^{2}+b^{2}}\leq\frac{1}{2\sqrt{2}\zeta(2\beta)}.
\]
Consequently, for any $u=(u_{1},u_{2})\in L_{2}^{\infty}$, there exists a
unique solution $x_{u}\in D((-\Delta)^{\beta})$ of the problem%
\[
\left\{
\begin{array}
[c]{l}%
(-\Delta)^{\beta}x_{1}(t)=a\sin(x_{2}(t))+t^{-\frac{1}{3}}e^{u_{1}(t)}\\
(-\Delta)^{\beta}x_{2}(t)=b\cos(x_{1}(t))+tu_{2}(t)
\end{array}
\right.
\]
for $t\in(0,\pi)$ a.e., and the mapping $\lambda(u)=(x_{u}^{1},x_{u}^{2})$ is
continuously differentiable with the differential $\lambda^{\prime}%
(u):L_{2}^{\infty}\rightarrow D((-\Delta)^{\beta})$ such that%
\begin{multline*}
(-\Delta)^{\beta}(\lambda^{\prime}(u)v)(t)-\left[
\begin{array}
[c]{cc}%
0 & a\cos((x_{u})_{2}(t))\\
-b\sin((x_{u})_{1}(t)) & 0
\end{array}
\right]  (\lambda^{\prime}(u)v)(t)\\
=\left[
\begin{array}
[c]{cc}%
t^{-\frac{1}{3}}e^{u_{1}(t)} & 0\\
0 & tu_{2}(t)
\end{array}
\right]  v(t),\ t\in(0,\pi)\text{ a.e.,}%
\end{multline*}
for any $v\in L_{2}^{\infty}$, i.e.%
\[
\left\{
\begin{array}
[c]{c}%
(-\Delta)^{\beta}(\left(  \lambda^{\prime}(u)v\right)  _{1})(t)=a\cos
((x_{u})_{2}(t))\left(  \lambda^{\prime}(u)v\right)  _{2}(t)+t^{-\frac{1}{3}%
}e^{u_{1}(t)}v_{1}(t)\\
(-\Delta)^{\beta}(\left(  \lambda^{\prime}(u)v\right)  _{1})(t)=-b\sin
((x_{u})_{1}(t))\left(  \lambda^{\prime}(u)v\right)  _{1}(t)+tu_{2}(t)v_{2}(t)
\end{array}
\right.
\]
for$\ t\in(0,\pi)$ a.e., and any $v=(v_{1},v_{2})\in L_{2}^{\infty}.$
\end{example}


\begin{thebibliography}{99}                                                                                               %


\bibitem {Alex}A. Alexiewicz, \textit{Functional Analysis}, PWN, Warsaw, 1969
(in Polish).

\bibitem {Buttazzo}H. Attouch, G. Buttazzo, G. Michaille, \textit{Variational
Analysis in Sobolev and BV Spaces. Applications to PDEs and Optimization},
SIAM-MPS, Philadelphia, 2006.

\bibitem {Barrios}B. Barrios, E. Colorado, A. de Pablo, U. S\'{a}nchez,
\textit{On some critical problems for the fractional Laplacian operator},
Journal of Differential Equations 252 (2012), 6133-6162.

\bibitem {BogBycz2}K. Bogdan, T. Byczkowski, \textit{Potential theory of
Schr\"{o}dinger operator based on fractional Laplacian}, Probability and
Mathematical Statistics 20 (2) (2000), 293-335.

\bibitem {Bors1}D. Bors, \textit{Global solvability of Hammerstein equations
with applications to BVP involving fractional Laplacian}, Abstract and Applied
Analysis 2013, Article Id 240863, 10 pages; DOI: 10.1155/2013/240863.

\bibitem {Bors2}D. Bors, \textit{Stability of nonlinear Dirichlet BVPs
governed by fractional Laplacian}, The Scientific World Journal 2014, Article
ID 920537, 10 pages; DOI: 10.1155/2014/920537.

\bibitem {Brezis}H. Brezis, \textit{Functional Analysis, Sobolev Spaces and
Partial Differential Equations}, Springer, New York, 2011.

\bibitem {CabreTan}X. Cabr\'{e}, J. Tan, \textit{Positive solutions of
nonlinear problems involving the square root of the Laplacian}, Adv. Math. 224
(2010) 2052--2093.

\bibitem {Caff}L. A. Caffarelli, \textit{Further regularity for the Signorini
problem}, Communications in Partial Differential Equations, vol. 4, no. 9
(1979), 1067-1075.

\bibitem {Id}D. Idczak, \textit{A global implicit function theorem and its
applications to functional equations}, Discrete and Continuous Dynamical
Systems, Series B 19 (8) (2014), 2549-2556, DOI: 10.3934/dcdsb.2014.19.2549.

\bibitem {Id2}D. Idczak, \textit{On a generalization of a global implicit
function theorem}, Advanced Nonlinear Studies 16 (1) (2016); DOI: 10.1515/ans-2015-5008.

\bibitem {Id3}D. Idczak, \textit{A bipolynomial fractional Dirichlet-Laplace
problem}, submitted for publication.

\bibitem {ISW}D. Idczak, A. Skowron, S. Walczak, \textit{On the
diffeomorphisms between Banach and Hilbert spaces}, Advanced Nonlinear Studies
12 (2012), 89-100.

\bibitem {Mlak}W. Mlak, \textit{An Introduction to the Hilbert Space Theory},
PWN, Warsaw, 1970 (in Polish).

\bibitem {Vaz}J. J. Vazquez, \textit{Nonlinear diffusion with fractional
Laplacian operators}, Nonlinear Partial Differential Equations, vol. 7 of Abel
Symposia (2012), 271-298.
\end{thebibliography}
\end{document}